\documentclass[a4paper,11pt]{article}

\usepackage{amsmath,amsfonts,amssymb,amsthm,dsfont,bbm,mathrsfs}
\usepackage{geometry}
\usepackage{paralist}
\usepackage{enumitem}
\usepackage[english]{babel}
\usepackage{textcase} 
\usepackage{filemod}
\usepackage[longnamesfirst]{natbib}
\usepackage{etoolbox} 

\makeatletter

\def\cite{\citet}

\theoremstyle{plain}
\newtheorem{theorem}{Theorem}[section]

\newtheorem{corollary}[theorem]{Corollary}

\theoremstyle{definition}

\numberwithin{equation}{section}
\allowdisplaybreaks

\AfterEndEnvironment{theorem}{\@afterheading}
\AfterEndEnvironment{lemma}{\@afterheading}
\AfterEndEnvironment{definition}{\@afterheading}
\AfterEndEnvironment{lemma}{\@afterheading}
\AfterEndEnvironment{corollary}{\@afterheading}
\AfterEndEnvironment{remark}{\@afterheading}
\AfterEndEnvironment{example}{\@afterheading}
\AfterEndEnvironment{proof}{\@afterheading}

\renewcommand\section{\@startsection {section}{1}{\z@}%
{-3.5ex \@plus -1ex \@minus -.2ex}%
{1.3ex \@plus.2ex}%
{\center\small\sc\MakeTextUppercase}}

\def\subsection#1{\@startsection {subsection}{2}{0pt}%
{-3.5ex \@plus -1ex \@minus -.2ex}%
{1ex \@plus.2ex}%
{\bf\mathversion{bold}}{#1}}

\def\subsubsection#1{\@startsection{subsubsection}{3}{0pt}%
{\medskipamount}%
{-10pt}%
{\normalsize\itshape}{\kern-2.2ex. #1.}}

\def\be#1{\begin{equation*}#1\end{equation*}}
\def\ben#1{\begin{equation}#1\end{equation}}
\def\bes#1{\begin{equation*}\begin{split}#1\end{split}\end{equation*}}
\def\besn#1{\begin{equation}\begin{split}#1\end{split}\end{equation}}

\def\given{\mskip 0.5mu plus 0.25mu\vert\mskip 0.5mu plus 0.15mu}
\newcounter{bracketlevel}%
\def\@bracketfactory#1#2#3#4#5#6{%
\expandafter\def\csname#1\endcsname##1{%
\global\advance\c@bracketlevel 1\relax%
\global\expandafter\let\csname @middummy\alph{bracketlevel}\endcsname\given%
\global\def\given{\mskip#5\csname#4\endcsname\vert\mskip#6}\csname#4l\endcsname#2##1\csname#4r\endcsname#3%
\global\expandafter\let\expandafter\given\csname @middummy\alph{bracketlevel}\endcsname%
\global\advance\c@bracketlevel -1\relax%
}%
}
\def\bracketfactory#1#2#3{%
\@bracketfactory{#1}{#2}{#3}{relax}{0.5mu plus 0.25mu}{0.5mu plus 0.15mu}
\@bracketfactory{b#1}{#2}{#3}{big}{1mu plus 0.25mu minus 0.25mu}{0.6mu plus 0.15mu minus 0.15mu}
\@bracketfactory{bb#1}{#2}{#3}{Big}{2.4mu plus 0.8mu minus 0.8mu}{1.8mu plus 0.6mu minus 0.6mu}
\@bracketfactory{bbb#1}{#2}{#3}{bigg}{3.2mu plus 1mu minus 1mu}{2.4mu plus 0.75mu minus 0.75mu}
\@bracketfactory{bbbb#1}{#2}{#3}{Bigg}{4mu plus 1mu minus 1mu}{3mu plus 0.75mu minus 0.75mu}
}
\bracketfactory{clc}{\lbrace}{\rbrace}
\bracketfactory{clr}{(}{)}
\bracketfactory{cls}{[}{]}
\bracketfactory{abs}{\lvert}{\rvert}
\bracketfactory{norm}{\Vert}{\Vert}
\bracketfactory{floor}{\lfloor}{\rfloor}
\bracketfactory{ceil}{\lceil}{\rceil}
\bracketfactory{angle}{\langle}{\rangle}

\newcounter{ctr}\loop\stepcounter{ctr}\edef\X{\@Alph\c@ctr}%
	\expandafter\edef\csname s\X\endcsname{\noexpand\mathscr{\X}}
	\expandafter\edef\csname c\X\endcsname{\noexpand\mathcal{\X}}
	\expandafter\edef\csname b\X\endcsname{\noexpand\boldsymbol{\X}}
	\expandafter\edef\csname I\X\endcsname{\noexpand\mathbbm{\X}}
\ifnum\thectr<26\repeat

\let\@IE\IE\let\IE\undefined
\newcommand{\IE}{\mathop{{}\@IE}\mathopen{}}
\let\@IP\IP\let\IP\undefined
\newcommand{\IP}{\mathop{{}\@IP}}
\newcommand{\Var}{\mathop{\mathrm{Var}}}

\newcommand{\law}{\mathop{{}\sL}\mathopen{}}

\let\original@left\left
\let\original@right\right
\renewcommand{\left}{\mathopen{}\mathclose\bgroup\original@left}
\renewcommand{\right}{\aftergroup\egroup\original@right}

\def\^#1{\relax\ifmmode {\mathaccent"705E #1} \else {\accent94 #1} \fi}
\def\~#1{\relax\ifmmode {\mathaccent"707E #1} \else {\accent"7E #1} \fi}
\def\*#1{\relax#1^\ast}
\edef\-#1{\relax\noexpand\ifmmode {\noexpand\bar{#1}} \noexpand\else \-#1\noexpand\fi}
\def\>#1{\vec{#1}}
\def\.#1{\dot{#1}}

\def\atop{\@@atop}
\def\%#1{\mathcal{#1}}

\renewcommand{\leq}{\leqslant}
\renewcommand{\geq}{\geqslant}
\renewcommand{\phi}{\varphi}
\newcommand{\eps}{\varepsilon}

\newcommand{\N}{\mathop{{}\mathrm{N}}}
\newcommand{\eq}{\eqref}
\newcommand{\I}{\mathop{{}\mathrm{I}}\mathopen{}}

\newcommand{\dw}{\mathop{d_{\mathrm{W}}}\mathopen{}}
\newcommand{\dk}{\mathop{d_{\mathrm{K}}}\mathopen{}}
\newcommand\indep{\protect\mathpalette{\protect\@indep}{\perp}}
\def\@indep#1#2{\mathrel{\rlap{$#1#2$}\mkern2mu{#1#2}}}

\newcommand{\toinf}{\to\infty}

\def\parsetime#1#2#3#4#5#6{#1#2:#3#4}
\def\parsedate#1:20#2#3#4#5#6#7#8+#9\empty{20#2#3-#4#5-#6#7 \parsetime #8}
\def\moddate{\expandafter\parsedate\pdffilemoddate{\jobname.tex}\empty}

\makeatother

\begin{document}

\title{\sc\bf\large\MakeUppercase{On quantitative bounds in the~mean~martingale~central limit theorem}}
\author{\sc Adrian R\"ollin}

\date{\itshape National University of Singapore}

\maketitle

\begin{abstract}
\noindent We provide explicit bounds on the Wasserstein distance between discrete time martingales and the standard normal distribution. The proofs are based on a combination of Lindeberg's and Stein's method.
\end{abstract}

\section{Introduction}

Let $X_1,\dots,X_n$ be a martingale difference sequence, that is, a sequence of random variables adapted to a filtration $\cF_0,\dots,\cF_n$ such that
\ben{\label{1}
    \IE(X_k|\cF_{k-1}) = 0 \qquad\text{almost surely for $1\leq k\leq n$,}
}
and let
\be{
   S_0:=0,\qquad S_k := X_1+\dots+X_k \quad\text{for $ 1\leq k\leq n$,}
}
be the resulting discrete-time martingale. For $1\leq k\leq n$, define the quantities
\be{
  \sigma^2_k = \IE(X_{k}^2|\cF_{k-1}),
  \qquad   
  \-\sigma^2_k := \IE X_{k}^2 = \IE\sigma^2_k,  
  \qquad 
  s^2_n := \Var S_n = \sum_{i=1}^n \bar\sigma_i^2,
}
and for $0\leq k\leq n$ define
\be{
  V_k^2 := \sum_{i=1}^k \sigma^2_i,
  \qquad
  \rho^2_{k+1} :=V_n^2 - V_{k}^2 =  \sum_{i=k+1}^n \sigma^2_i,
}
where $\sum_{i=a}^b$ is defined to be zero if $a>b$.

The asymptotic behaviour of $S_n/s_n$ has already been intensively studied for many decades, probably starting with \cite{Billingsley1961} and \cite{Ibragimov1963}, who proved central limit theorem in certain special cases; we refer to the classical textbook \cite{Hall1980}. Bounds on the quality of normal approximation were also obtained, such as by \cite{Ibragimov1963} and \cite{Heyde1970} with respect to the Kolmogorov distance, on which most later work has focused, too. For example, it was shown by \cite{Bolthausen1982a} that rates of order $n^{-1/4}$ as $n\toinf$ are sharp if uniformly bounded third moments are assumed, even under the strong assumption that the conditional variances satisfy $\sigma_k^2 = \bar\sigma_k^2$ almost surely. For uniformly bounded random variables and assuming only that 
\ben{\label{2}
    V^2_n=s_n^2 \quad\text{almost surely,}
}
\cite{Bolthausen1982a} improved the rate to $n^{-1/2}\log n$, for which he again showed that it is sharp. Various embellishments were obtained later; see, for example, \cite{Haeusler1988} and \cite{ElMachkouri2007} to name but a few. 

In contrast, bounds with respect to the Wasserstein distance are rare. To the best of our knowledge, the first result was obtained by \cite{Dedecker2008} in the case of stationary martingale differences, and later generalised by \cite{Van-Dung2014} under conditions akin to those asserted by \cite{Bolthausen1982a}. 

In the standard literature, proofs to obtain quantitative bounds in the martingale central limit theorems are often based on Lindeberg's telescoping sum argument. The individual differences in the sum are usually handled by Taylor expansion, followed by some sort of smoothing argument to obtain bounds with respect to the Kolmogorov or Wasserstein distances. As of now, there are no proofs based on Stein's method (\cite{Stein1972}), and the work on these notes was started with the intention to close this gap.
While we were not able to find a proof purely based on Stein's method, our basic approach is instead a combination of \emph{both} Lindeberg's and Stein's method. Our proofs also start with Lindeberg's telescoping sum, but we then use Stein's method to handle the individual differences in the sum. This seems to by-pass the tedious smoothing arguments appearing in many of the aforementioned articles, resulting in shorter proofs under weaker assumptions. In order not to just provide new proofs of already known results, and also in order to keep things simple, we restrict ourselves to the Wasserstein distance and consider conditions similar to those of \cite{Bolthausen1982a}.

\section{Main results}

Let $\dw(F,G)$ denote the Wasserstein distance between two distributions $F$ and $G$ on the real line, which are assumed to have finite first moments. This distance is defined as $\dw(F,G)=\sup_{h}\abs{\int hdF - \int hdG }$, where the supremum ranges over all $1$-Lipschitz-continuous functions $h$. For distributions on the real line, we can alternatively write $\dw(F,G)=\int\abs{F(x)-G(x)}dx$; see \cite{Vallender1973}. The following is our main result, from which we then deduce various corollaries.

\begin{theorem}  \label{thm1}
Assume that $V_n^2 = s_n^2$ almost surely. Then, for any $a\geq 0$,
\ben{                                              \label{3}
    \dw\bclr{\law(S_n/s_n),\N(0,1)} \leq   
    \frac{3}{s_n}
    \sum_{k=1}^n\IE\frac{\abs{X_k}^3}{\rho_k^2+a^2} + \frac{2a}{s_n}.
}
\end{theorem}

\begin{proof} Let $Z',Z_1,\dots,Z_n,$ be a sequence of independent standard normal random variables, also independent of $\cF_n$, and let
\ben{\label{4}
   Z := \sum_{i=1}^n \sigma_i Z_i,\qquad
   T_k:=\sum_{i=k}^n\sigma_i Z_i, \quad 1\leq k\leq n+1.
}
Note that $Z$ is normally distributed with variance $s_n^2$, both conditionally on $\cF_n$ and unconditionally. Moreover, note that both
\ben{\label{5}
    \text{$\rho_{k}^2$ and $\rho_{k+1}^2$ are $\cF_{k-1}$-measurable;}
}
the latter is true since, by assumption, $\rho^2_{k+1} = V^2_n-V^2_{k} = s^2_n-V^2_{k}$ almost surely and $V^2_{k}$ is $\cF_{k-1}$-measurable, and since $\rho_k^2=\sigma_k^2+\rho_{k+1}^2$ the former also follows. From \eq{5}, we conclude that
\ben{\label{6}
    \text{$\law(T_k|\cF_{k-1}) \sim \N(0,\rho_k^2)$\quad and \quad $\law(T_{k+1}|\cF_{k-1}) \sim \N(0,\rho_{k+1}^2)$}
}
almost surely.
Now, fix a $1$-Lipschitz-continuous function $h$, and note that $h$ is differentiable almost everywhere; denote by $h'$ this derivative and extend it to the whole real line, for example, by using the left derivative. We clearly have $\norm{h'}\leq 1$, where $\norm{\cdot}$ denotes the supremum norm. Using the triangle inequality, it is easy to see that
\ben{\label{7}
  \abs{\IE\clc{h(S_n) - h(Z)}}
  \leq\abs{\IE\clc{h(S_n+aZ') - h(Z+aZ')}} + 2a.
}
Then, using Lindeberg's telescoping sum representation and conditioning inside the expectation, write
\ben{\label{8}
    \IE\clc{h(S_n+aZ') - h(Z+aZ')}
     = \IE\, \sum_{k=1}^n 
        \IE\clc{R_k\given \cF_{k-1}},
}
where 
\ben{\label{9} 
  R_k=h(S_k + T_{k+1}+aZ') - h(S_{k-1} + T_k+aZ').
}
Let $g$ be the unique bounded solution to
\be{
     g'(x)-xg(x) = \~h(x) - \IE \~h(Y),
    \qquad x\in\IR,
}
where $Y\sim \N(0,1)$ and where $\~h$ is any measurable real-valued function. In the case where~$\~h$ is Lipschitz-continuous, \cite{Stein1986} and \cite{Raic2004} proved that $\norm{g''}\leq 2\norm{\~h'}$ and $\norm{g'}\leq \sqrt{2/\pi}\norm{\~h'}\leq \norm{\~h'}$, respectively.
Let $s\in\IR$, let $t>0$, and let $\~h(x) =  h\clr{tx+s}/t$, which is Lipschitz-continuous. Defining $f_{s,t}(w) := g\clr{(w-s)/t}$, where $w\in\IR$, it is not difficult to see that $f_{s,t}$ satisfies
\ben{\label{10}
    t^2 f_{s,t}'(w)-(w-s)f_{s,t}(w) = h(w) - \IE h(tY+s),
    \qquad w\in\IR.
}
From the bounds on $g$, we easily obtain the bounds
\ben{\label{11}
      \bnorm{f'_{s,t}}\leq \frac{\norm{h'}}{t},
        \qquad\bnorm{f''_{s,t}}\leq \frac{2\norm{h'}}{t^2}.
}
It is also straightforward to show that $f_{s,t}(w)$, as well as $f'_{s,t}(w)$ and $f''_{s,t}(w)$, understood as functions from $\IR\times\IR_{>0}\times\IR\to\IR$, are measurable, so that in what follows, we are allowed to write expressions like $f_{U,V}(W)$ for arbitrary random variables~$U$, $V$ and $W$, where $V>0$. 

Now, let $T_{k}':= T_{k} + aZ'$, and note that by \eq{6},
\ben{\label{11b}
  \law\bclr{T_{k}'\given\cF_{k-1} } = \N(0,\rho^2_k+a^2).
}
We now consider everything conditionally on $\cF_{k-1}$; let $s=S_{k-1}$ and $t = \rho'_k := \sqrt{\rho_k^2+a^2}$, and note that both $s$ and $t$ are $\cF_{k-1}$-measurable. Hence, using the definition of $T'_k$, then~\eq{11b}, and then~\eq{10} with $w$ being replaced by $S_k+T'_{k+1}$, we can write
\besn{\label{12}
   & \IE\clc{R_k\given\cF_{k-1}} \\
   &\quad =  \IE\clc{h(S_k + T_{k+1}+aZ') - h(S_{k-1} + T_k+aZ')\given\cF_{k-1}}\\   
   &\quad =  \IE\clc{h(S_k + T_{k+1}') - h(S_{k-1} +T'_k)\given\cF_{k-1}}\\   
   &\quad =  \IE\clc{h(S_k + T_{k+1}') - h( \rho'_kY+S_{k-1})\given\cF_{k-1}}\\   
   &\quad =  
    \IE\bclc{\rho_k'^{2} f_{S_{k-1},\rho_k'}'(S_k+T_{k+1}') -
(S_k+T_{k+1}'-S_{k-1})f_{S_{k-1},\rho_k'}(S_k+T_{k+1}')\given\cF_{k-1} }.
}
Moreover, recalling the definition of $\rho'_k$ and recalling that $\rho_k^2 = \sigma^2_k+ \rho_{k+1}^2$,
\bes{
    & \IE\bclc{\rho_k'^{2} f_{S_{k-1},\rho_k'}'(S_k+T_{k+1}') -
(S_k+T_{k+1}'-S_{k-1})f_{S_{k-1},\rho_k'}(S_k+T_{k+1}')\given\cF_{k-1} }\\
    & \quad= \IE\bclc{(\rho_k^2+a^2) f_{S_{k-1},\rho_k'}'(S_k+T_{k+1}') -
(X_k+T_{k+1}')f_{S_{k-1},\rho_k'}(S_k+T_{k+1}')\given\cF_{k-1} }\\
    & \quad = \IE\bclc{\sigma_k^2 f_{S_{k-1},\rho_k'}'(S_k+T_{k+1}') -
X_k f_{S_{k-1},\rho_k'}(S_k+T_{k+1}')\given\cF_{k-1} }\\
    & \quad\qquad +\IE\bclc{(\rho_{k+1}^2+a^2) f_{S_{k-1},\rho_k'}'(S_k+T_{k+1}') -
T_{k+1}'f_{S_{k-1},\rho_k'}(S_k+T_{k+1}')\given\cF_{k-1} }\\
    &\quad =  
    \IE\bclc{\sigma_k^2 f_{S_{k-1},\rho_k'}'(S_k+T_{k+1}') -
X_k f_{S_{k-1},\rho_k'}(S_k+T_{k+1}')\given\cF_{k-1} },
}
where in order to obtain the last equality, we used the fact that $\law\bclr{T_{k+1}'\given\cF_{k-1} } = \N(0,\rho^2_{k+1}+a^2)$ and that $\IE\clc{b^2 g'(Y) - Y g(Y)} = 0$ for every function $g$ for which the expectation exist whenever $Y\sim\N(0,b^2)$, with $b = \rho_{k+1}^2+a^2$.

Using Taylor expansion, the fact that $\IE\clc{X_k\given\cF_{k-1}}=0$ and  $\IE\clc{X_k^2\given\cF_{k-1}}=\sigma_k^2$, and also that, conditionally on $\cF_{k-1}$, $X_k$ and $T_{k+1}'$ are independent of each other,  we obtain
\besn{\label{13}
  \IE\clc{R_k\given\cF_{k-1}}
    &= \bbbabs{\IE\bbbclc{\sigma^2_k X_k\int_0^1
 f_{S_{k-1},\rho_k'}''(S_{k-1}+sX_k+T_{k+1}')ds \\
    &\kern3em+ X_k^3\int_0^1(1-s)
	 f_{S_{k-1},\rho_k'}''(S_{k-1}+sX_k+T_{k+1}')ds\given \cF_{k-1}}}.
}
Using \eq{11},
\be{
  \abs{\IE\clc{R_k\given\cF_{k-1}}}
    \leq1.5\bnorm{f_{S_{k-1},\rho_k'}''}\IE\bclc{\abs{X_k}^3\given \cF_{k-1}}
    \leq
    \frac{3\norm{h'}}{\rho_k^2+a^2}\IE\bclc{\abs{X_k}^3\given \cF_{k-1}}.
}
Thus,
\ben{
    \abs{\IE h(S_n)-\IE h(Z)} \leq   
    3
    \sum_{k=1}^n\IE\frac{\abs{X_k}^3}{\rho_k^2+a^2} +  2a.
}
Scaling by $1/s_n$, the final bound follows.
\end{proof}

The following corollary is an immediate consequence of Theorem~\ref{thm1}; it not only gives a better rate of convergence under weaker conditions than Theorem~4 of \cite{Van-Dung2014}, but also explicit constants.

\begin{corollary}\label{cor1} 
Assume that $V_n^2 = s_n^2$ almost surely, and assume there exist constants~$\alpha$ and $\gamma$ such that
$0<\alpha\leq\sigma^2_k$ and $\IE\abs{X_k}^3\leq\gamma$ for  $1\leq k \leq n$. Then
\ben{\label{15}
    \dw\bclr{\law(S_n/s_n),\N(0,1)} \leq
\frac{3\gamma(1+\log n)}{\alpha^{3/2}\sqrt{n}} .
}
\end{corollary}

As mentioned before, a rate of order $n^{1/4}$ is
sharp for the Kolmogorov distance under the conditions of Corollary~\ref{cor1}. It is not difficult to see
that for any random variable $X$, 
\be{
    \dw\bclr{\law(X),\N(0,1)}\leq \eps
    \quad\text{implies} \quad
    \dk\bclr{\law(X),\N(0,1)}\leq \eps^{1/2},
}
where $\dk$ denotes the Kolmogorov distance,
so that \eq{17} comes quite close to the optimal rate, but has the advantage
of giving explicit constants.

If the conditional variances can not be bounded away from zero, we need stronger conditions on the third moments. Indeed, \cite[Theorem~2]{Bolthausen1982a} assumes in this case uniformly bounded $X_i$. Using Theorem~\ref{thm1}, it will be enough
to assume some appropriate bounds on the conditional third moments.

\begin{corollary}\label{cor2} Assume $V_n^2=s_n^2$ almost surely, and assume there exist constants~$\beta$ and~$\delta$ such that 
\ben{                                                 \label{16}
    \IE\bclc{\abs{X_k}^3\given\cF_{k-1}}\leq \beta\wedge\delta\sigma^2_k,
    \qquad 1\leq k\leq n.
}
Then
\ben{                                                 \label{17}
    \dw\bclr{\law(S_n/s_n),\N(0,1)}
    \leq \frac{3\delta n(s^2_n/n+\beta^{2/3})(1+\log n)}{s_n^3} 
    +\frac{2}{\sqrt{n}}.
}
\end{corollary}

\begin{proof}
Define the sequence of stopping times
\be{
    \tau_0:=0,
    \qquad\tau_k:=\sup\bclc{m\geq 0:V^2_m\leq s^2_nk/n}
        \quad\text{for $1\leq k < n$},
    \qquad\tau_n:=n.
}
Note that $\{\tau_k = m\} = \bclc{V^2_{m}\leq s^2_nk/n}\cap \bclc{V^2_{m+1}> s^2_nk/n}$, and since both $V^2_{m}$ and $V^2_{m+1}$ are $\cF_m$-measurable, it follows that $\{\tau_k = m\}\in\cF_{m}$, so that $\tau_k$ is indeed a stopping time. 
Now, if $j\leq\tau_k$, we have
\be{
    \rho^2_{j} = s_n^2 - V_{j-1}^2\geq s_n^2(n-k)/n.
}
Thus, for $1\leq k \leq n$,
\besn{\label{18}
    \IE\sum_{j=\tau_{k-1}+1}^{\tau_{k}}
        \frac{\abs{X_j}^3}{\rho^2_j+s_n^2/n}
    & = \IE\sum_{j=1}^n\IE\bbbclc{ \frac{\abs{X_j}^3}{\rho^2_j+s^2_n/n}\I[\tau_{k-1}< j\leq \tau_k]\given \cF_{j-1}}\\
    & \leq\frac{n}{s^2_n(n-k+1)} \IE\sum_{j=1}^n\I[\tau_{k-1}< j\leq \tau_k]\IE\bclc{\abs{X_j}^3\given \cF_{j-1}}\\
    & \leq
    \frac{\delta n}{s_n^2(n-k+1)}\IE(V_{\tau_{k}}^2-V_{\tau_{k-1}}^2)
    \leq \frac{\delta n(s^2_n/n+\beta^{2/3})}{s_n^2(n-k+1)},
}
where we have used \eq{16} in the second-last and last inequality. With
\be{
    \sum_{k=1}^n\IE
        \frac{\abs{X_k}^3}{\rho^2_k+s^2_n/n}
    \leq
    \sum_{k=1}^{n}\IE\sum_{j=\tau_{k-1}+1}^{\tau_{k}}
        \frac{\abs{X_j}^3}{\rho^2_j+s^2_n/n} 
    \leq 
     \sum_{k=1}^{n} \frac{\delta n(s^2_n/n+\beta^{2/3})}{s_n^2(n-k+1)},
}
the final bound now easily follows from Theorem~\ref{thm1} with $a=s_n/\sqrt{n}$.
\end{proof}

The convergence behaviour of $S_n/s_n$ to the normal distribution is intimately
connected to the behaviour of $V_n^2/s_n^2$ and thus rates of convergence of
$S_n$ crucially depend on the rate of convergence of
\ben{\label{19}
    V_n^2 /s_n^2 \to 1 \qquad \text{as $n\toinf$.}
}

\begin{corollary} 
Assume there exist constants~$\beta$ and~$\delta$ such that 
\eq{16} holds. Then
\bes{
    &\dw\bclr{\law(S_n/s_n),\N(0,1)}\\
    &\qquad\leq 1.5\bclr{\IE\babs{V_n^2/s_n^2-1}}^{1/2} + 
     \frac{3 n(1.6\beta^{1/3}\vee \delta)(s^2_n/n+1.4\beta^{2/3})(1+\log n)}{s_n^3} 
    +\frac{2}{\sqrt{n}}.
} 

\end{corollary}

\begin{proof} Define the stopping time $\tau=\sup\bclc{m\leq n: V_m^2\leq s_n^2}$, and let the new martingale difference sequence~$\~X_1,\dots,\~X_{2n}$ be defined as follows.
For $1\leq k\leq\tau$, let $\~X_k := X_k$. Define
$R:=\floor{(s_n^2-V^2_\tau)\beta^{-2/3}}$ and note that $\tau+R$ is also a stopping time. Let
$\~X_{\tau+1},\dots,\~X_{\tau+R}$ be i.i.d.\/ with distribution~$\N(0,\beta^{2/3})$, let $\~X_{\tau+R+1}$ have distribution $\N(0,s_n^2-V_\tau^2-R\beta^{2/3})$, and set $\~X_k := 0$ for all $\tau+R+2\leq k\leq 2n$. Note that, because $s^2_n\leq n\beta^{2/3}$, we will never need more than~$2n$ indices. By construction, we have $\~V^2_{2n} = s^2_n$ almost surely, and
\be{
    \IE\bclc{\abs{\~X_k}^3\given\cF_{k-1}} 
    \leq 
   \begin{cases}
     \beta\wedge\delta\sigma^2_k
      & \text{if $1\leq k\leq\tau$,} \\
     1.6\beta\wedge 1.6\beta^{1/3}\sigma_k^2
      & \text{if $\tau < k \leq \tau+R$,} \\
     1.6(s_n^2-R\beta^{2/3})^{3/2}\wedge 1.6(s_n^2-R\beta^{2/3})^{1/2}\sigma^2_k
      & \text{if $k=\tau+R+1$,} \\ 
     0 \wedge 0\sigma_k^2
      & \text{if $\tau+R+2\leq k\leq 2n$.}
   \end{cases}
}
Since $s_n^2-R\beta^{2/3}\leq \beta^{3/2}$, it is follows that
\be{
  \IE\bclc{\abs{\~X_k}^3\given\cF_{k-1}}
  \leq 1.6\beta \wedge (1.6\beta^{1/3}\vee \delta)\sigma^2_k,
  \qquad 1\leq k\leq 2n,
}
so that we can apply Corollary~\ref{cor2} to $\~S_{2n}$ with $\beta$ being replaced by $1.6\beta$ and $\delta$ being replaced by $1.6\beta^{1/3}\vee \delta$.
Noting moreover that 
\be{
    \IE\abs{S_n-\~S_{2n}} 
    \leq \bbbcls{\IE\bbbclr{\sum_{i=\tau}^n X_i - 
            \sum_{i=\tau+1}^{\tau+R+1} \~X_i}^2}^{1/2}
    \leq \sqrt{2}\,\bclr{\IE\babs{V^2_n-s^2_n}}^{1/2},
}
the claim follows.
\end{proof}

It is possible to replace the random quantities $\rho_k^2$ in Theorem~\ref{thm1} by the averaged quantities $\-\rho_k^2 = \IE\rho_k^2$.

\begin{theorem}\label{thm2} For any martingale with finite
third moments, 
\ben{                                                    \label{20}
    \dw\bclr{\law(S_n/s_n),\N(0,1)} \leq   
    \frac{1}{s_n}
    \sum_{k=1}^n\frac{3\IE\abs{X_k}^3+\IE\abs{\sigma^2_k-\-\sigma^2_k}}{\-\rho^2_k+a^2} + \frac{2a}{s_n}.
} 
\end{theorem}

\begin{proof}
The proof is analogous to that of Theorem~\ref{thm1}. Define                
$Z := \sum_{i=1}^n \bar\sigma_i Z_i$ and 
$T_k:=\sum_{i=k}^n\bar\sigma_i Z_i$, where we now use unconditional variances instead of conditional variances. Analogously to \eq{13}, we can show that
\bes{
  \IE\clc{R_k\given\cF_{k-1}}
    &= \bbbabs{\IE\bbbclc{
    \bclr{\bar\sigma^2_k-\sigma^2_k}f'_{S_{k-1},\-\rho_k'}(S_{k-1}+sX_k+T_{k+1}')\\
     &\kern4em+
    \sigma^2_k X_k\int_0^1
 f_{S_{k-1},\-\rho_k'}''(S_{k-1}+sX_k+T_{k+1}')ds \\
    &\kern4em+ X_k^3\int_0^1(1-s)
	 f_{S_{k-1},\-\rho_k'}''(S_{k-1}+sX_k+T_{k+1}')ds\given \cF_{k-1}}}.
}
The rest of the argument runs similarly to the proof of Theorem~\ref{thm1}.
\end{proof}

\section*{Acknowledgements}

I thank Dalibor Voln\'y for suggesting the problem of finding a proof of the martingale CLT via Stein's method, and Fang Xiao for encouraging me to publish these notes, which were written in~2008. The present manuscript represents a corrected and updated version of my original notes. I also thank the anonymous referee for helpful comments, which have further improved presentation and readability.

\setlength{\bibsep}{0.5ex}
\def\bibfont{\small}


\end{document}